\documentclass[a4paper]{article} 

\usepackage{a4wide}
\usepackage{authblk}
\usepackage{amsmath,amssymb,amsthm}
\usepackage{latexsym}
\usepackage{xcolor}
\usepackage{graphicx}
\usepackage{natbib}

\newcommand{\pp}{{\mathbb P}}
\newcommand{\RR}{{\mathbb R}}
\newcommand{\C}[1]{C^{#1}}
\newcommand{\bracket}[1]{\left( #1 \right)}
\newcommand{\set}[1]{\left\{#1\right\}}
\newcommand{\tset}[1]{\{#1\}}
\newcommand{\defset}[2]{\set{#1:\,#2}}
\newcommand{\tri}{\Delta}
\newcommand{\PSR}[1]{#1_{\mathrm{PS}}}
\newcommand{\PST}{\PSR{\tri}}

\newcommand{\Tsym}{T_{\mathrm{S}}}
\newcommand{\splspace}[3]{{\mathbb S}_{#1}^{#2}(#3)}
\newcommand{\splspaceb}{\mathcal{S}}
\newcommand{\splspacedb}{\splspaceb{}^*}
\newcommand{\redsplspace}[3]{{\mathbb S}_{#1}^{\mathrm{R}}(#3)}
\newcommand{\redsplspaceb}{\mathcal{S}_{\mathrm{R}}}
\newcommand{\bloss}[1]{\mathcal{B}[#1]}
\newcommand{\blossr}[2]{\bloss{{#1}|_{#2}}}

\newtheorem{theorem}{Theorem}
\newtheorem{lemma}[theorem]{Lemma}
\newtheorem{proposition}[theorem]{Proposition}
\newtheorem{example}[theorem]{Example}
\newtheorem{remark}[theorem]{Remark}

\begin{document}


\title{Using Geometric Symmetries to Achieve Super-Smoothness for Cubic Powell--Sabin Splines}

\author[$\dagger$,$\ddag$]{Jan Gro\v{s}elj}
\affil[$\dagger$]{Faculty of Mathematics and Physics, University of Ljubljana, Slovenia}
\affil[$\ddag$]{Institute of Mathematics, Physics and Mechanics, Ljubljana, Slovenia}
\author[$\star$]{Hendrik Speleers}
\affil[$\star$]{Department of Mathematics, University of Rome Tor Vergata, Italy}

\date{}

\maketitle

\begin{abstract}\noindent
In this paper, we investigate $\C{2}$ super-smoothness of the full $\C{1}$ cubic spline space on a Powell--Sabin refined triangulation, for which a B-spline basis can be constructed. 
Blossoming is used to identify the $\C{2}$ smoothness conditions between the functionals of the dual basis. Some of these conditions can be enforced without difficulty on general triangulations. Others are more involved but greatly simplify if the triangulation and its corresponding Powell--Sabin refinement possess certain symmetries.
Furthermore, it is shown how the $\C{2}$ smoothness constraints can be integrated into the spline representation by reducing the set of basis functions. 
As an application of the super-smooth basis functions, a reduced spline space is introduced that maintains the cubic precision of the full $\C{1}$ spline space.
\end{abstract}


\section{Introduction}

Smooth polynomial spline spaces over general triangulations are known for their dependency on the mesh geometry which complicates the identification of degrees of freedom, entangles the analysis of splines, and often jeopardizes their stability. One of the standard approaches for constructing splines of this type is by using macro-element techniques: the above notorious problems can be avoided by applying additional restrictions to the general definition, e.g., extra refinement of the triangulation, super-smoothness in certain regions of the domain, elimination of derivative degrees of freedom, and combinations of these constraints. An important feature of macro-element methods is that splines are constructed locally over each triangle of the triangulation with little or no dependence on the neighbouring triangles.

A majority of macro-element techniques have emerged from the finite element method for numerically solving partial differential equations via the variational approach (see, e.g., \cite{fem_ciarlet_02}). The geometric interpretation of these techniques in terms of the Bernstein--B\'ezier representation has simplified the analysis of splines and enabled many of their generalizations, especially to higher smoothness orders (see, e.g., \cite{lai_07}). Recently, efforts have been made to describe such spline spaces by B-spline basis functions which automatically incorporate smoothness conditions, are locally supported, form a nonnegative partition of unity, and are distinguished by excellent stability properties.

A macro-element technique that has proven to be well suited for the development of B-spline methods is the refinement of the triangulation by the Powell--Sabin 6-split \cite{ps_powell_77}. For the standard $\C{1}$ quadratic splines over such a refinement, a B-spline basis was constructed in \cite{ps_dierckx_97} and influenced several extensions to higher polynomial degrees and smoothness orders. Among these are the specifically tailored B-spline representations of $\C{2}$ quintic splines \cite{ps5_speleers_10,ps5_speleers_12} and almost $\C{2}$ quartic splines \cite{ps4_barrera_21,ps4_groselj_16} as well as more general constructions of B-spline bases for splines of arbitrary smoothness \cite{ps_speleers_13,ps_speleers_15} and degree \cite{ps_groselj_16}.

Special attention has been lately dedicated to $\C{1}$ cubic Powell--Sabin splines. B-spline representations for specific subspaces were presented in \cite{ps3_lamnii_14} and \cite{ps3_speleers_15}. A B-spline basis for the full space with certain geometric constraints on the underlying triangulation was constructed in~\cite{ps3_groselj_16}. Finally, a nonrestrictive $\C{1}$ B-spline representation was introduced in \cite{ps3_groselj_17}. The latter provides a general framework for describing specific super-smooth cubic Powell--Sabin splines, quadratic Powell--Sabin splines, and even cubic Clough--Tocher splines \cite{ct_clough_65}. The potential of this framework in approximation of functions and numerical solution of partial differential equations was explored in \cite{ps3_groselj_18} and \cite{ps3_groselj_23}.

This paper investigates $\C{2}$ smoothness constraints on $\C{1}$ cubic splines over a Powell--Sabin refinement by exploiting their B-spline representation and its dual characterization derived in \cite{ps3_groselj_17}. The dual basis consists of functionals that can be expressed via blossoming of the polynomials obtained by restricting a spline to certain triangles of the refinement. Since smoothness constraints between polynomials can be elegantly described by blossoming identities, this offers a powerful tool for the study of super-smoothness that is more transparent than resolving relations acquired through the Bernstein--B\'ezier methods.

It is known from \cite{ps3_groselj_17,ps3_groselj_18} that the B-spline representation of $\C{1}$ cubic Powell--Sabin splines allows a straightforward imposition of certain $\C{2}$ smoothness constraints inside triangles of the original triangulation.
In this paper, we further analyze smoothness constraints of such splines and, in order to conceal their macro-structure, aim at the full $\C{2}$ smoothness inside a macro-triangle.
We show that this can be locally achieved when the triangle 6-split is symmetric in the sense that certain edges of the refinement lie on the same line. This result is not practically irrelevant as triangulations are often at least partially structured and as such allow symmetric refinement of a possibly significant number of triangles.

This paper also addresses the implementation of the $\C{2}$ smoothness conditions into the B-spline representation by recombining the original B-spline functions and reducing their number. The procedure resembles the one presented in \cite{ps3_groselj_21} for three-directional triangulations. It offers much more versatility, however, as it can be used on triangulations that are only partially structured. This is thanks to a more general recombination technique which permits the transition from the original B-spline representation in regions that are partitioned in an unstructured way to the new reduced B-spline representation on symmetrically refined parts of the triangulation.

Moreover, the reduction process ensures that the new basis functions still reproduce cubic polynomials, which is crucial for obtaining optimal convergence rates in approximation methods. In this way a better accuracy with respect to the number of degrees of freedom can be achieved as it has been recently demonstrated in \cite{ps3_groselj_23} for uniformly refined triangulations --- these are a specific setting in which the methods presented in this paper are applicable. The reduced B-spline basis also preserves all other important features of the original one such as the properties of local support, nonnegativity, and partition of unity, making it a viable option for computer aided geometric design.

The remainder of the paper is organized as follows. In Section~\ref{sec:c1_space} we review the Powell--Sabin refinement and the properties of the B-spline representation of $\C{1}$ cubic Powell--Sabin splines. In Section~\ref{sec:c2_super_smooth} we study $\C{2}$ super-smoothness conditions for a spline in such a representation. Section~\ref{sec:ss_basis} is dedicated to the reduction of the B-spline functions, and Section~\ref{sec:reduced_space} presents the reduced spline space spanned by the derived basis functions. We conclude with some remarks in Section~\ref{sec:conclusion}.

Throughout the paper, we use the following convention in our notation regarding indices. When we write $a_{ij}$ for some indices $i$ and $j$, it means that this quantity is uniquely determined up to permutation of the indices, so $a_{ij}=a_{ji}$. On the other hand, when we separate the indices by commas, it means that the order of the indices is important, so $a_{i,j}$ and $a_{j,i}$ are different quantities for some indices $i$ and~$j$. Similarly, when we write $a_{i,jk}$ for some indices $i$, $j$, and $k$, it means that this is equivalent to $a_{i,kj}$ but different from, e.g., $a_{k,ij}$.

\section{$C^1$ Spline Space}
\label{sec:c1_space}

In this section we review how a given triangulation $\tri$ of a planar domain is refined into a Powell--Sabin triangulation $\PST$, and how a space of $\C{1}$ cubic splines over $\PST$ can be represented in terms of basis functions mimicking univariate B-splines.

\subsection{Powell--Sabin Refinement}

A triangulation $\Delta = (V, E, T)$ of a polygonal domain $\Omega \subset \RR^2$ is determined by a finite set of vertices $V = \set{v_i}_i$, a finite set of edges $E = \set{e_{ij}}_{ij}$, and a finite set of triangles $T = \set{t_{ijk}}_{ijk}$. Every vertex $v_i \in V$ corresponds to a point in $\Omega$, every edge $e_{ij} \in E$ corresponds to a convex hull $[v_i, v_j]$ of two different vertices $v_i, v_j \in V$, and every triangle $t_{ijk} \in T$ corresponds to a convex hull $[v_i, v_j, v_k]$ of three noncollinear vertices $v_i, v_j, v_k \in V$. In this notation, we identify the edges and triangles by the same set of indices, i.e., the order of indices does not matter. We assume that the union of triangles in $T$ equals $\Omega$ and the intersection of any two triangles from $T$ is an empty set, an element of $V$, or an element of $E$. Naturally, we require that every $v_i \in V$ is a vertex and every $e_{ij} \in E$ is an edge of at least one triangle in $T$. We denote by $E^b \subseteq E$ the set of all edges that lie on the boundary of $\Omega$.

A Powell--Sabin refinement $\PST = (\PSR{V}, \PSR{E}, \PSR{T})$ of $\tri$ is a triangulation obtained by splitting each triangle $t_{ijk} \in T$ into six smaller triangles \cite{lai_07,ps_powell_77}. This is achieved by choosing a triangle split point $v_{ijk}$ inside $t_{ijk}$ such that for an interior edge $e_{ij} \in E \setminus E^b$ that is shared by triangles $t_{ijk}, t_{ijk'} \in T$, the intersection point $v_{ij}$ between $e_{ij}$ and $[v_{ijk}, v_{ijk'}]$
lies inside $e_{ij}$. The point $v_{ij}$ is an edge split point. For a boundary edge $e_{ij} \in E^b$, we choose the edge split point $v_{ij}$ as an arbitrary point inside $e_{ij}$. Together with the vertices from $V$, these split points constitute $\PSR{V}$. For each $t_{ijk} \in T$ we define the triangles
\begin{alignat*}{6}
t_{i,j,k} &= [v_i, v_{ij}, v_{ijk}],\quad & t_{j,k,i} &= [v_j, v_{jk}, v_{ijk}],\quad & t_{k,i,j} &= [v_k, v_{ki}, v_{ijk}], \\
t_{j,i,k} &= [v_j, v_{ij}, v_{ijk}], & t_{k,j,i} &= [v_k, v_{jk}, v_{ijk}], & t_{i,k,j} &= [v_i, v_{ki}, v_{ijk}].
\end{alignat*}
These triangles constitute $\PSR{T}$. Notice that here the indices are separated by commas, indicating that their order is important. Moreover, notice that $t_{i,j,k}\subset t_{ijk}$ in this notation. For each $t_{ijk} \in T$ we also define the edges
\begin{alignat*}{6}
e_{i,j} &= [v_i, v_{ij}],\quad & e_{j,k} &= [v_j, v_{jk}],\quad & e_{k,i} &= [v_k, v_{ki}], \\
e_{j,i} &= [v_j, v_{ij}], & e_{k,j} &= [v_k, v_{jk}], & e_{i,k} &= [v_i, v_{ki}].
\end{alignat*}
Again, notice that the indices are separated by commas and that $e_{i,j}\subset e_{ij}$.
Together with $[v_i, v_{ijk}]$, $[v_j, v_{ijk}]$, $[v_k, v_{ijk}]$ and $[v_{ij}, v_{ijk}]$, $[v_{jk}, v_{ijk}]$, $[v_{ki}, v_{ijk}]$, these edges constitute $\PSR{E}$. We denote by $\PSR{E}^b \subset \PSR{E}$ the set of all edges that lie on the boundary of $\Omega$.

\begin{figure}[t!]
\centering
\includegraphics[width=\textwidth]{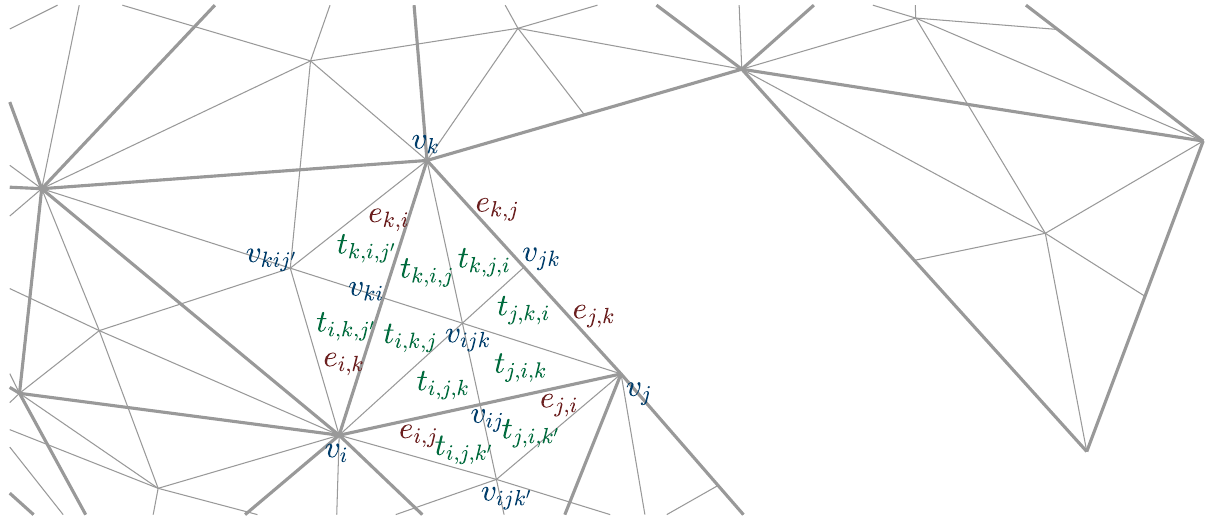}
\caption{An excerpt from a triangulation and its Powell--Sabin refinement.}
\label{fig:pstri}
\end{figure}

For any triangulation $\tri$, there exists a valid triangulation $\PST$. It can be constructed by choosing $v_{ijk}$ as the incenter of $t_{ijk} \in T$. An example of a Powell--Sabin refinement of a triangulation is shown in Figure~\ref{fig:pstri}. The figure also illustrates some notation for the vertices, edges, and triangles of the refinement.

\subsection{Blossoming of a Cubic Polynomial}

Let $\pp_3$ denote the space of bivariate polynomials of total degree at most $3$. The blossom of a polynomial $P \in \pp_3$ is a mapping $\bloss{P}: \RR^2 \times \RR^2 \times \RR^2 \rightarrow \RR$ that is
\begin{itemize}

\item symmetric, i.e., for every $p_1, p_2, p_3 \in \RR^2$, and any permutation $\pi$ of $\set{1,2,3}$,
\begin{equation*}
\bloss{P}(p_1, p_2, p_3) = \bloss{P}(p_{\pi(1)}, p_{\pi(2)}, p_{\pi(3)});
\end{equation*}

\item affine in the first (and hence each) argument, i.e., for every $p, q, p_2, p_3 \in \RR^2$ and every $\alpha \in \RR$,
\begin{equation*}
\bloss{P}((1-\alpha) p + \alpha q, p_2, p_3) = (1-\alpha) \bloss{P}(p, p_2, p_3) + \alpha \bloss{P}(q, p_2, p_3);
\end{equation*}

\item and for every $p \in \RR^2$,
\begin{equation*}
\bloss{P}(p, p, p) = P(p).
\end{equation*}
\end{itemize}
By the blossoming principle, for each $P \in \pp_3$, the mapping $\bloss{P}$ is unique. We refer the reader to, e.g., \cite{seidel_93,bb_speleers_11} for more details.

Blossoming is a convenient tool for expressing smoothness conditions of joints of two or more polynomials.

\begin{lemma}
\label{lem:smoothness_edge}
Suppose that $t_1, t_2 \subset \RR^2$ are two nondegenerate and nonoverlapping triangles with a common edge $e$. Let $S: t_1 \cup t_2 \rightarrow \RR$ be a $\C{0}$ smooth spline
\begin{equation*}
S(p) =
\begin{cases}
P_1(p) & \text{if } p \in t_1, \\
P_2(p) & \text{if } p \in t_2,
\end{cases}
\end{equation*}
determined by polynomials $P_1, P_2 \in \pp_3$. Let $p_1, p_2 \in \RR^2$ be any pair of different points on the line containing $e$.

\begin{enumerate}

\item Let $q \in \RR^2$ be any point not on the line containing $e$. Then, $S$ is $C^1$ smooth across $e$ if and only if%
\begin{subequations}%
\label{eq:c1_smooth_blossom}
\begin{align}
\bloss{P_1}(p_1, p_1, q) &= \bloss{P_2}(p_1, p_1, q), \\
\bloss{P_1}(p_1, p_2, q) &= \bloss{P_2}(p_1, p_2, q), \\
\bloss{P_1}(p_2, p_2, q) &= \bloss{P_2}(p_2, p_2, q).
\end{align}
\end{subequations}

\item Suppose that $S$ is $\C{1}$ smooth, and let $q_1, q_2 \in \RR^2$ be any pair of (possibly equal) points not on the line containing $e$. Then, $S$ is $C^2$ smooth across $e$ if and only if%
\begin{subequations}%
\label{eq:c2_smooth_blossom}
\begin{align}
\bloss{P_1}(p_1, q_1, q_2) &= \bloss{P_2}(p_1, q_1, q_2), \\
\bloss{P_1}(p_2, q_1, q_2) &= \bloss{P_2}(p_2, q_1, q_2).
\end{align}
\end{subequations}

\end{enumerate}
\end{lemma}
\begin{proof}
According to \cite[Proposition~21.1]{ramshaw_87}, the spline $S$ is $\C{r}$ smooth, $r \in \set{1,2}$, across $e$ if and only if
\begin{equation}
\label{eq:smooth_blossom}
\bloss{P_1}(a_1, a_2, a_3) = \bloss{P_2}(a_1, a_2, a_3)
\end{equation}
for every three points $a_1, a_2, a_3 \in \RR^2$ of which at least $3-r$ lie on the line containing $e$.

First, we consider $r = 1$. Clearly, if $S$ is $\C{1}$ smooth, the equalities in \eqref{eq:c1_smooth_blossom} hold due to \eqref{eq:smooth_blossom}. Conversely, let $a_1, a_2, a_3 \in \RR^2$ be any three points. Since the blossom is symmetric and $S$ is $\C{0}$ smooth, we can assume that $a_1, a_2$ lie on the line containing $e$ and $a_3$ does not. Then, $a_1, a_2$ can be expressed as affine combinations of $p_1, p_2$ and $a_3$ as an affine combination of $p_1, p_2, q$. As the blossom is affine in each argument and $S$ is $\C{0}$ smooth, the equalities in \eqref{eq:c1_smooth_blossom} imply \eqref{eq:smooth_blossom}.

For $r = 2$, we argue similarly. The $\C{2}$ smoothness of $S$ clearly implies \eqref{eq:c2_smooth_blossom}. Conversely, it is sufficient to consider a point $a_1 \in \RR^2$ lying on the line containing $e$ and points $a_2, a_3 \in \RR^2$ not lying on this line. Then, $a_1$ can be expressed as an affine combination of $p_1, p_2$ and $a_2$ as an affine combination of $p_1, p_2, q_1$ and $a_3$ as an affine combination of $p_1, p_2, q_2$. As the blossom is affine in each argument and $S$ is $\C{1}$ smooth, the equalities in \eqref{eq:c2_smooth_blossom} imply \eqref{eq:smooth_blossom}.
\end{proof}

\subsection{Basis Functions}
\label{sec:c1_basis}

The space of $\C{1}$ cubic splines over $\PST$ is defined by
\begin{equation*}
\splspace{3}{1}{\PST} = \defset{S \in \C{1}(\Omega)}{S|_{t_{i,j,k}} \in \pp_3, \ t_{i,j,k} \in \PSR{T}}
\end{equation*}
and its dimension is equal to $3|V| + 4|E| = 3|V| + |\PSR{T}| + |\PSR{E}^b|$; see \cite{ps3_groselj_17}. Therein, a basis of the space suitable for numerical analysis has been derived. In order to express a spline $S \in \splspace{3}{1}{\PST}$ in that basis, we introduce three types of functionals.

For each vertex $v_i \in V$, let $q_i^v = [q_{i,1}^v, q_{i,2}^v, q_{i,3}^v] \subset \RR^2$ be a triangle that contains $v_i$, the point $\frac{2}{3} v_i + \frac{1}{3} v_{ijk}$ for every $t_{ijk} \in T$ with a vertex at $v_i$, and the point $\frac{2}{3} v_i + \frac{1}{3} v_{ij}$ for every $e_{ij} \in E$ with an endpoint at $v_i$; see Figure~\ref{fig:ps3conf} for an illustration. Then, for each $r \in \set{1,2,3}$, let%
\begin{subequations}%
\label{eq:beta}
\begin{equation}
\label{eq:beta_v}
\beta_{i,r}^v(S) = \blossr{S}{t_{i,j,k}}(v_i, v_i, v_i + 3 (q_{i,r}^v - v_i)),
\end{equation}
where $t_{i,j,k} \in \PSR{T}$ is any triangle with a vertex at $v_i$. The functional $\beta_{i,r}^v$ is well defined because $S \in \C{1}(\Omega)$; see Lemma~\ref{lem:smoothness_edge} and the equalities in \eqref{eq:c1_smooth_blossom}. 

\begin{figure}[t!]
\centering
\includegraphics[width=\textwidth]{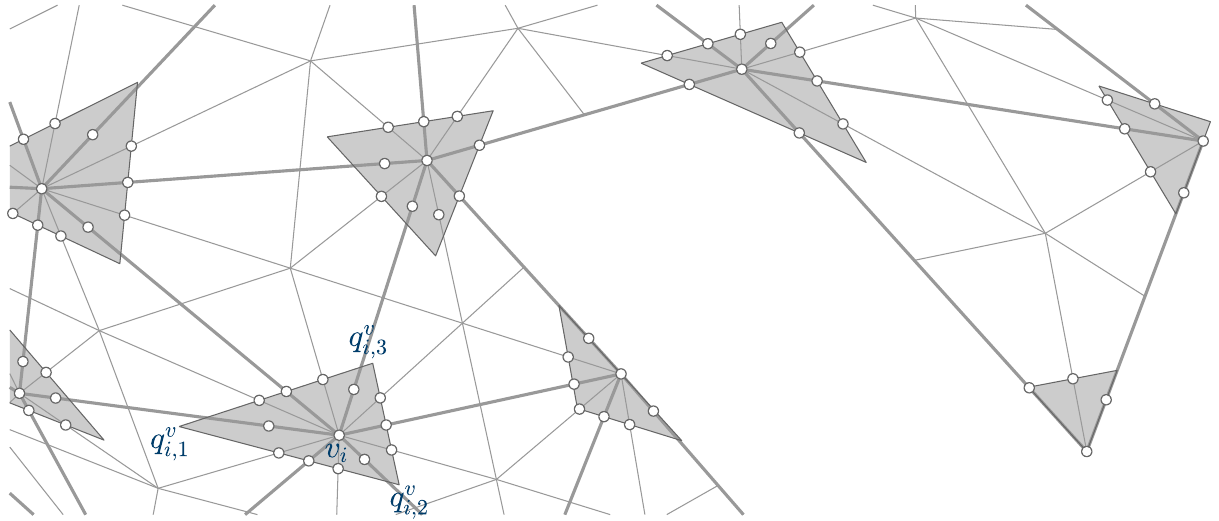}
\caption{Triangles (in gray) associated with vertices of a triangulation. The figure also depicts the points that each of the triangles must contain.}
\label{fig:ps3conf}
\end{figure}

Moreover, for each triangle $t_{i,j,k} \in \PSR{T}$, let
\begin{equation}
\label{eq:beta_t}
\beta_{i,j,k}^t(S) = \blossr{S}{t_{i,j,k}}(v_i, v_j, v_{ijk}),
\end{equation}
and for each boundary edge $e_{i,j} \in \PSR{E}^b$, let
\begin{equation}
\label{eq:beta_e}
\beta_{i,j}^e(S) = \blossr{S}{e_{i,j}}(v_i, v_j, v_{ij}).
\end{equation}
\end{subequations}
We denote the set of functionals defined in \eqref{eq:beta} by $\splspacedb$.

In \cite{ps3_groselj_17}, the following basis functions were constructed.
\begin{subequations}
\label{eq:c1_basis}
\begin{itemize}
\item For each $v_i \in V$ and each $r \in \set{1,2,3}$, let $B_{i,r}^v \in \splspace{3}{1}{\PST}$ be such that
\begin{equation}
\label{eq:c1_basis_v}
\beta(B_{i,r}^v) =
\begin{cases}
1 & \text{if} \ \beta = \beta_{i,r}^v, \\
0 & \text{otherwise},
\end{cases}
\end{equation}
for every $\beta \in \splspacedb$; see Figure~\ref{fig:c1_basis_v} for an example.
\item For each $t_{i,j,k} \in \PSR{T}$, let $B_{i,j,k}^t \in \splspace{3}{1}{\PST}$ be such that
\begin{equation}
\label{eq:c1_basis_t}
\beta(B_{i,j,k}^t) =
\begin{cases}
1 & \text{if} \ \beta = \beta_{i,j,k}^t, \\
0 & \text{otherwise},
\end{cases}
\end{equation}
for every $\beta \in \splspacedb$; see Figure~\ref{fig:c1_basis_t} for an example.
\item For each boundary edge $e_{i,j} \in \PSR{E}^b$, let $B_{i,j}^e \in \splspace{3}{1}{\PST}$ be such that
\begin{equation}
\label{eq:c1_basis_e}
\beta(B_{i,j}^e) =
\begin{cases}
1 & \text{if} \ \beta = \beta_{i,j}^e, \\
0 & \text{otherwise},
\end{cases}
\end{equation}
for every $\beta \in \splspacedb$; see Figure~\ref{fig:c1_basis_e} for an example.
\end{itemize}
\end{subequations}
We denote the set of these functions by $\splspaceb$. The following theorem (see \cite{ps3_groselj_17} for its proof) provides a characterization of $\splspace{3}{1}{\PST}$.

\begin{figure}[t!]
\centering
\includegraphics[width = 0.32\textwidth]{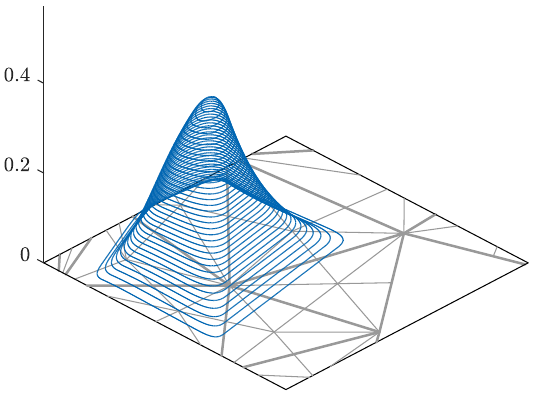}~
\includegraphics[width = 0.32\textwidth]{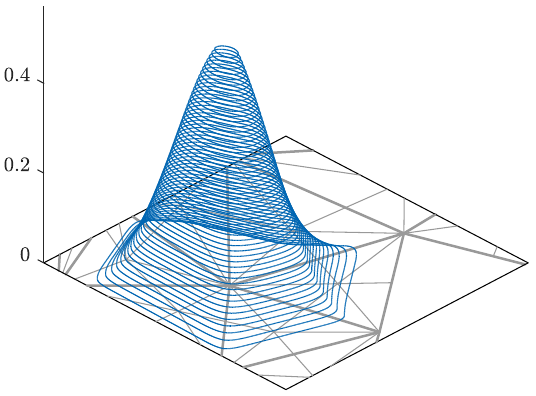}~
\includegraphics[width = 0.32\textwidth]{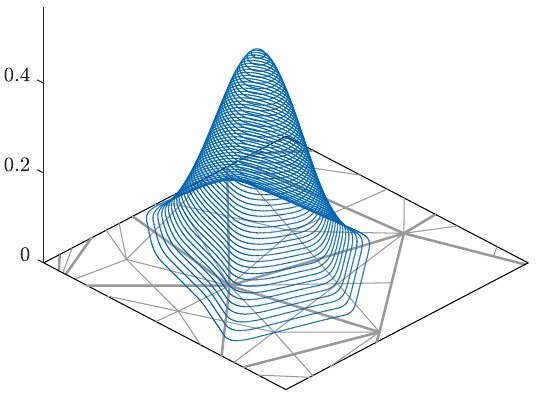}
\caption{The basis functions $B_{i,1}^v$, $B_{i,2}^v$, $B_{i,3}^v$ (from left to right) associated with the vertex $v_i$ and the points $q_{i,1}^v$, $q_{i,2}^v$, $q_{i,3}^v$ shown in Figure~\ref{fig:ps3conf}.}
\label{fig:c1_basis_v}
\end{figure}

\begin{figure}[t!]
\centering
\includegraphics[width = 0.49\textwidth]{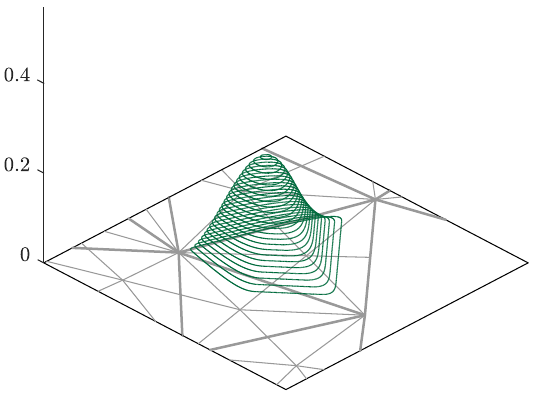}~
\includegraphics[width = 0.49\textwidth]{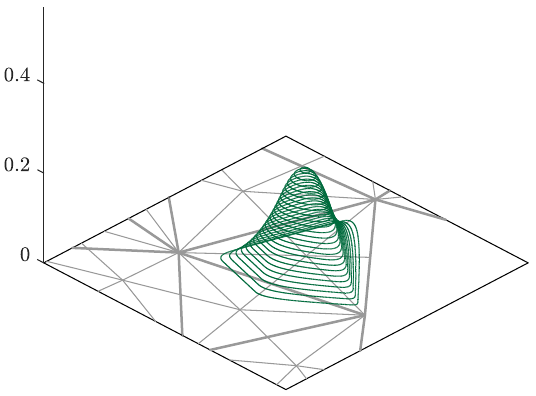}\\
\includegraphics[width = 0.49\textwidth]{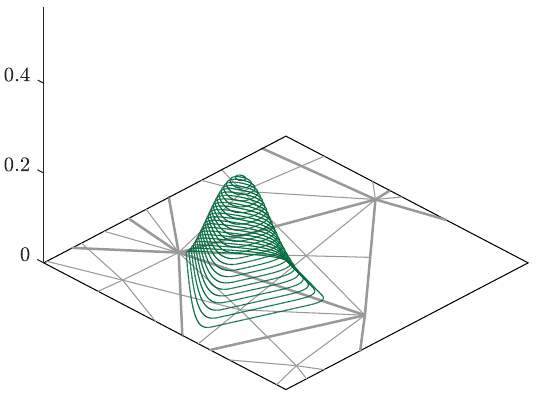}~
\includegraphics[width = 0.49\textwidth]{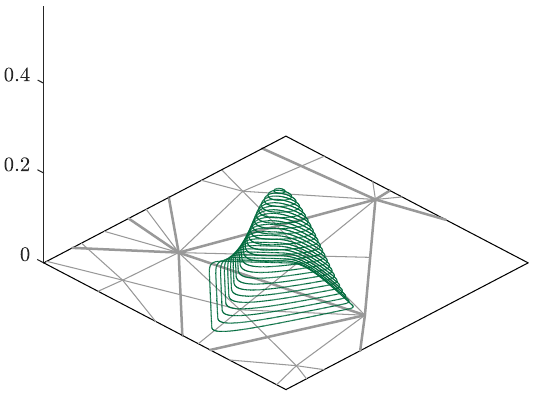}
\caption{The basis functions $B_{i,j,k}^t$, $B_{j,i,k}^t$ (in the upper row) and $B_{i,j,k'}^t$, $B_{j,i,k'}^t$ (in the bottom row) associated with the triangles $t_{i,j,k}$, $t_{j,i,k}$, $t_{i,j,k'}$, $t_{j,i,k'}$ shown in Figure~\ref{fig:pstri}.}
\label{fig:c1_basis_t}
\end{figure}

\begin{figure}[t!]
\centering
\includegraphics[width = 0.49\textwidth]{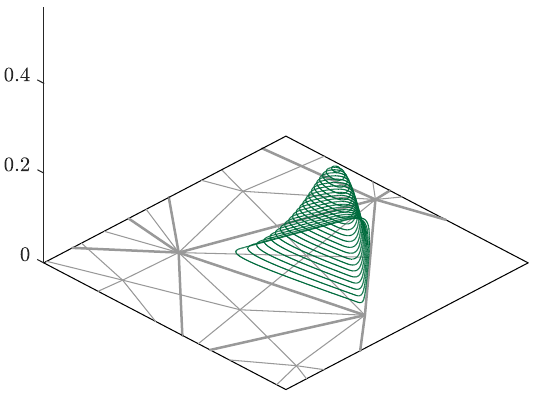}~
\includegraphics[width = 0.49\textwidth]{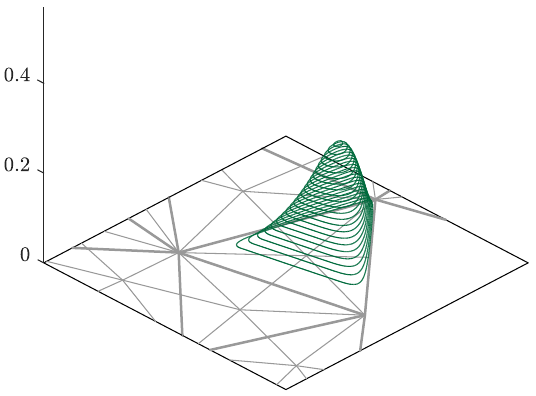}\\
\includegraphics[width = 0.49\textwidth]{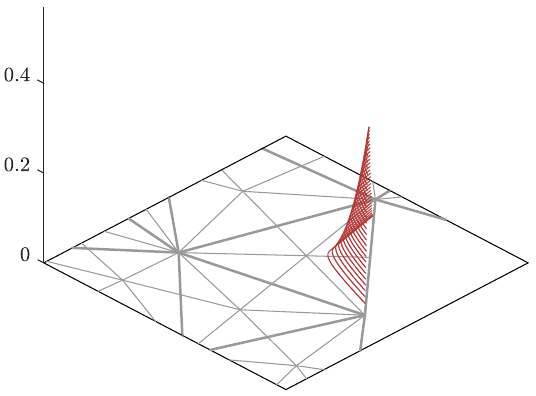}~
\includegraphics[width = 0.49\textwidth]{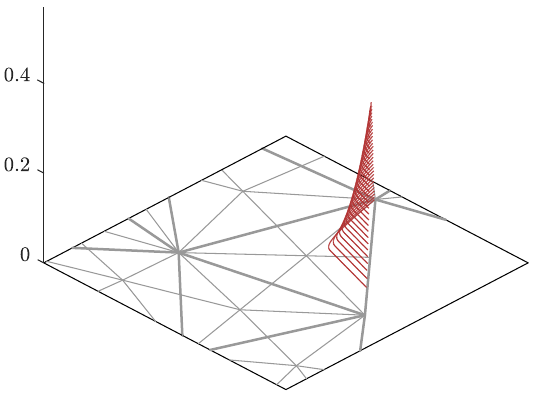}
\caption{The basis functions $B_{j,k,i}^t$, $B_{k,j,i}^t$ (in the upper row) and $B_{j,k}^e$, $B_{k,j}^e$ (in the bottom row) associated with the triangles $t_{j,k,i}$, $t_{k,j,i}$ and the edges $e_{j,k}$, $e_{k,j}$ shown in Figure~\ref{fig:pstri}.}
\label{fig:c1_basis_e}
\end{figure}

\begin{theorem}
\label{th:c1_basis}
The set $\splspaceb$ is a basis of $\splspace{3}{1}{\PST}$. Every $S \in \splspace{3}{1}{\PST}$ can be uniquely expressed as
\begin{equation*}
S = \sum_{v_i \in V} \sum_{r = 1}^3 \beta_{i,r}^v(S) B_{i,r}^v + \sum_{t_{i,j,k} \in \PSR{T}} \beta_{i,j,k}^t(S) B_{i,j,k}^t + \sum_{e_{i,j} \in \PSR{E}^b} \beta_{i,j}^e(S) B_{i,j}^e. 
\end{equation*}
\end{theorem}

Theorem~\ref{th:c1_basis} also implies that the functions from $\splspaceb$ form a partition of unity. Moreover, these functions are locally supported and nonnegative.
For more details regarding the properties of $\splspaceb$, we refer the reader to \cite{ps3_groselj_17}.

\begin{remark}
Actually, a more general construction of basis functions was considered in \cite{ps3_groselj_17}. The functionals in \eqref{eq:beta_t}--\eqref{eq:beta_e} were defined as follows: for each triangle $t_{i,j,k} \in \PSR{T}$ and some parameter $\sigma_{i,j,k}>0$,
\begin{equation*}
\beta_{i,j,k}^t(S) = \blossr{S}{t_{i,j,k}}(v_i, v_i+2\sigma_{i,j,k}(v_j-v_i), v_{ijk}),
\end{equation*}
and for each boundary edge $e_{i,j} \in \PSR{E}^b$ and some parameter $\sigma_{i,j}>0$,
\begin{equation*}
\beta_{i,j}^e(S) = \blossr{S}{e_{i,j}}(v_i, v_i+2\sigma_{i,j}(v_j-v_i), v_{ij}).
\end{equation*}
For simplicity, in this paper we assume all $\sigma_{i,j,k} = \sigma_{i,j} = \frac{1}{2}$.
\end{remark}

\section{$\C{2}$ Super-Smoothness}
\label{sec:c2_super_smooth}

Given a triangle $t_{ijk} \in T$ of $\tri$, we seek the conditions for a spline $S \in \splspace{3}{1}{\PST}$ to be $C^2$ in the interior of $t_{ijk}$. It suffices to derive the conditions for the $\C{2}$ smoothness of $S$ across the edges $[v_{ij}, v_{ijk}]$ and $[v_i, v_{ijk}]$ of $\PST$.

\subsection{$\C{2}$ Smoothness Across $[v_{ij}, v_{ijk}]$}

Let us consider the triangles $t_{i,j,k}, t_{j,i,k} \in \PSR{T}$ of $\PST$ attached to the edge $[v_{ij}, v_{ijk}] \in \PSR{E}$. By Lemma~\ref{lem:smoothness_edge}, the restriction of $S$ to $t_{i,j,k} \cup t_{j,i,k}$ is $\C{2}$ smooth across $[v_{ij}, v_{ijk}]$ if and only if%
\begin{subequations}%
\label{eq:c2_smooth_blossom_et}%
\begin{align}
\blossr{S}{t_{i,j,k}}(v_{ijk}, v_i, v_j) &= \blossr{S}{t_{j,i,k}}(v_{ijk}, v_i, v_j), \label{eq:c2_smooth_blossom_et_t} \\
\blossr{S}{t_{i,j,k}}(v_{ij}, v_i, v_j) &= \blossr{S}{t_{j,i,k}}(v_{ij}, v_i, v_j). \label{eq:c2_smooth_blossom_et_e}
\end{align}
\end{subequations}
More precisely, we obtain \eqref{eq:c2_smooth_blossom_et} from the equalities in \eqref{eq:c2_smooth_blossom} by choosing $P_1 = S|_{t_{i,j,k}}$, $P_2 = S|_{t_{j,i,k}}$, $p_1 = v_{ijk}$, $p_2 = v_{ij}$, and $q_1 = v_i$, $q_2 = v_j$. As the following proposition reveals, these two conditions can be expressed in terms of the dual functionals associated with certain triangles and boundary edges of $\PST$.

\begin{proposition}
\label{pro:c2_smoothness_et}
Let $e_{ij} \in E$ be an edge of $\tri$ and $t_{ijk} \in T$ a triangle of $\tri$ attached to $e_{ij}$. Suppose that $t_{ijk'} \in T$ is the second triangle of $\tri$ attached to $e_{ij}$ if $e_{ij} \in E \setminus E^b$ is an interior edge of $\tri$. Then, the restriction of a spline $S \in \splspace{3}{1}{\PST}$ to $t_{i,j,k} \cup t_{j,i,k}$ is $\C{2}$ smooth across $[v_{ij}, v_{ijk}]$ if and only if%
\begin{subequations}%
\label{eq:c2_smoothness_et}
\begin{equation}
\label{eq:c2_smoothness_et_t}
\beta_{i,j,k}^t(S) = \beta_{j,i,k}^t(S)
\end{equation}
and
\begin{equation}
\label{eq:c2_smoothness_et_e}
\left\{
\begin{aligned}
\beta_{i,j,k'}^t(S) &= \beta_{j,i,k'}^t(S) && \text{if} \ e_{ij} \in E \setminus E^b, \\
\beta_{i,j}^e(S) &= \beta_{j,i}^e(S) && \text{if} \ e_{ij} \in E^b.
\end{aligned}
\right.
\end{equation}
\end{subequations}
\end{proposition}
\begin{proof}
The condition \eqref{eq:c2_smooth_blossom_et_t} is equivalent to \eqref{eq:c2_smoothness_et_t} by the definitions of $\beta_{i,j,k}^t(S)$ and $\beta_{j,i,k}^t(S)$ in \eqref{eq:beta_t}. If $e_{ij} \in E^b$, the condition \eqref{eq:c2_smooth_blossom_et_e} is equivalent to \eqref{eq:c2_smoothness_et_e} by the definitions of $\beta_{i,j}^e(S)$ and $\beta_{j,i}^e(S)$ in \eqref{eq:beta_e}. Otherwise, if $e_{ij} \in E \setminus E^b$, there exists a value $\mu \in (0,1)$ such that
\begin{equation*}
v_{ij} = (1-\mu) v_{ijk} + \mu v_{ijk'}.
\end{equation*}
By the affinity of the blossom in the first argument and the $\C{1}$ smoothness of $S$ across $e_{i,j}$ and $e_{j,i}$,
\begin{align*}
\blossr{S}{t_{i,j,k}}(v_{ij}, v_i, v_j) = (1-\mu) \beta_{i,j,k}^t(S) + \mu \beta_{i,j,k'}^t(S), \\
\blossr{S}{t_{j,i,k}}(v_{ij}, v_i, v_j) = (1-\mu) \beta_{j,i,k}^t(S) + \mu \beta_{j,i,k'}^t(S).
\end{align*}
Assuming \eqref{eq:c2_smoothness_et_t}, \eqref{eq:c2_smooth_blossom_et_e} is then equivalent to \eqref{eq:c2_smoothness_et_e}.
\end{proof}

\begin{remark}
\label{rem:c2_smoothness_et}
The result presented in Proposition~\ref{pro:c2_smoothness_et} was already observed in \cite{ps3_groselj_17}. In addition, it was argued there that if for a spline $S \in \splspace{3}{1}{\PST}$ the conditions \eqref{eq:c2_smoothness_et} are satisfied for the edge $[v_{ij}, v_{ijk}]$ and analogous conditions are satisfied for the edges $[v_{jk}, v_{ijk}]$ and $[v_{ki}, v_{ijk}]$, then $S$ is also $\C{2}$ smooth at the split point $v_{ijk}$.
\end{remark}

\subsection{$\C{2}$ Smoothness Across $[v_i, v_{ijk}]$}

We now consider the triangles $t_{i,j,k}, t_{i,k,j} \in \PSR{T}$ of $\PST$ attached to the edge $[v_i, v_{ijk}] \in \PSR{E}$. Lemma~\ref{lem:smoothness_edge} implies that the restriction of a spline $S \in \splspace{3}{1}{\PST}$ to $t_{i,j,k} \cup t_{i,k,j}$ is $\C{2}$ smooth across $[v_i, v_{ijk}]$ if and only if%
\begin{subequations}%
\label{eq:c2_smooth_blossom_vt}
\begin{align}
\blossr{S}{t_{i,j,k}}(v_i, v_j, v_k) &= \blossr{S}{t_{i,k,j}}(v_i, v_j, v_k), \label{eq:c2_smooth_blossom_vt_v} \\
\blossr{S}{t_{i,j,k}}(v_{ijk}, v_j, v_k) &= \blossr{S}{t_{i,k,j}}(v_{ijk}, v_j, v_k). \label{eq:c2_smooth_blossom_vt_t}
\end{align}
\end{subequations}
More precisely, we obtain \eqref{eq:c2_smooth_blossom_vt} from the equalities in \eqref{eq:c2_smooth_blossom} by choosing $P_1 = S|_{t_{i,j,k}}$, $P_2 = S|_{t_{i,k,j}}$, $p_1 = v_i$, $p_2 = v_{ijk}$, and $q_1 = v_j$, $q_2 = v_k$.

In general, the blossom values in \eqref{eq:c2_smooth_blossom_vt_v} depend on the values of $\beta_{i,r}^v(S)$, $r = 1, 2, 3$. Consequently, they are not straightforward to satisfy when one considers the $\C{2}$ smoothness over several triangles of $\tri$ with a vertex at $v_i$. However, as the next observations show, they can be greatly simplified under certain geometric assumptions. 

\begin{lemma}
\label{lem:c2_smooth_vt_bloss_v}
Let $v_i \in V$ be a vertex of $\tri$ and $t_{ijk} \in T$ a triangle of $\tri$ with a vertex at $v_i$. Suppose that $t_{ijk'} \in T$ is the second triangle of $\tri$ attached to $e_{ij}$ if $e_{ij} \in E \setminus E^b$ is an interior edge of $\tri$. Moreover, suppose that the points $v_k$, $v_{ijk}$, $v_{ij}$ are collinear and $\nu_{ij,k} \in \RR$ is such that
\begin{equation}
\label{eq:nu}
v_k =
\begin{cases}
(1 - \nu_{ij,k}) v_{ijk} + \nu_{ij,k} v_{ijk'} & \text{if} \ e_{ij} \in E \setminus E^b, \\
(1 - \nu_{ij,k}) v_{ijk} + \nu_{ij,k} v_{ij} & \text{if} \ e_{ij} \in E^b.
\end{cases}
\end{equation}
Then,
\begin{equation*}
\blossr{S}{t_{i,j,k}}(v_i, v_j, v_k) = (1 - \nu_{ij,k}) \beta_{i,j,k}^t(S) + \nu_{ij,k}
\begin{cases}
\beta_{i,j,k'}^t(S) & \text{if} \ e_{ij} \in E \setminus E^b, \\
\beta_{i,j}^e(S) & \text{if} \ e_{ij} \in E^b,
\end{cases}
\end{equation*}
for every $S \in \splspace{3}{1}{\PST}$.
\end{lemma}
\begin{proof}
If $e_{ij} \in E \setminus E^b$, the affinity of the blossom in the third argument and the $\C{1}$ smoothness of $S$ across $e_{i,j}$ imply that $\blossr{S}{t_{i,j,k}}(v_i, v_j, v_k)$ can be decomposed into
\begin{equation*}
(1 - \nu_{ij,k}) \blossr{S}{t_{i,j,k}}(v_i, v_j, v_{ijk}) + \nu_{ij,k} \blossr{S}{t_{i,j,k'}}(v_i, v_j, v_{ijk'}).
\end{equation*}
Otherwise, if $e_{ij} \in E^b$, the affinity of the blossom in the third argument implies that $\blossr{S}{t_{i,j,k}}(v_i, v_j, v_k)$ can be decomposed into
\begin{equation*}
(1 - \nu_{ij,k}) \blossr{S}{t_{i,j,k}}(v_i, v_j, v_{ijk}) + \nu_{ij,k} \blossr{S}{t_{i,j,k}}(v_i, v_j, v_{ij}).
\end{equation*}
This confirms the statement by the definitions of $\beta_{i,j,k}^t(S)$ and $\beta_{i,j,k'}^t(S)$ in \eqref{eq:beta_t} and $\beta_{i,j}^e(S)$ in \eqref{eq:beta_e}.
\end{proof}

\begin{lemma}
\label{lem:c2_smooth_vt_bloss_t}
Let $v_i \in V$ be a vertex of $\tri$ and $t_{ijk} \in T$ a triangle of $\tri$ with a vertex at $v_i$. Suppose that the points $v_k$, $v_{ijk}$, $v_{ij}$ are collinear. Then,
\begin{equation*}
\blossr{S}{t_{i,j,k}}(v_{ijk}, v_j, v_k) = \beta_{j,k,i}^t(S),
\end{equation*}
for every $S \in \splspace{3}{1}{\PST}$.
\end{lemma}
\begin{proof}
Under the assumption of collinearity, it follows from the $\C{1}$ smoothness of $S$ across $[v_{ij}, v_{ijk}]$ that
\begin{equation*}
\blossr{S}{t_{i,j,k}}(v_{ijk}, v_j, v_k) = \blossr{S}{t_{j,i,k}}(v_{ijk}, v_j, v_k).
\end{equation*}
Then, by the $\C{1}$ smoothness of $S$ across $[v_j, v_{ijk}]$,
\begin{equation*}
\blossr{S}{t_{j,i,k}}(v_{ijk}, v_j, v_k) = \blossr{S}{t_{j,k,i}}(v_{ijk}, v_j, v_k).
\end{equation*}
The latter is equal to $\beta_{j,k,i}^t(S)$; see \eqref{eq:beta_t}.
\end{proof}

\begin{proposition}
\label{pro:c2_smoothness_vt}
Let $v_i \in V$ be a vertex of $\tri$ and $t_{ijk} \in T$ a triangle of $\tri$. Suppose that
\begin{itemize}
\item the points $v_k$, $v_{ijk}$, $v_{ij}$ are collinear and $t_{ijk'} \in T$ is the second triangle of $\tri$ attached to $e_{ij}$ if $e_{ij} \in E \setminus E^b$ is an interior edge;
\item the points $v_j$, $v_{ijk}$, $v_{ki}$ are collinear and $t_{kij'} \in T$ is the second triangle of $\tri$ attached to $e_{ki}$ if $e_{ki} \in E \setminus E^b$ is an interior edge. 
\end{itemize}
Then, the restriction of a spline $S \in \splspace{3}{1}{\PST}$ to $t_{i,j,k} \cup t_{i,k,j}$ is $\C{2}$ smooth across the edge $[v_i, v_{ijk}]$ if and only if
\begin{align*}
& (1 - \nu_{ij,k}) \beta_{i,j,k}^t(S) + \nu_{ij,k}
\begin{cases}
\beta_{i,j,k'}^t(S) & \text{if} \ e_{ij} \in E \setminus E^b \\
\beta_{i,j}^e(S) & \text{if} \ e_{ij} \in E^b
\end{cases} \\
&= (1 - \nu_{ki,j}) \beta_{i,k,j}^t(S) + \nu_{ki,j}
 \begin{cases}
\beta_{i,k,j'}^t(S) & \text{if} \ e_{ki} \in E \setminus E^b, \\
\beta_{i,k}^e(S) & \text{if} \ e_{ki} \in E^b,
\end{cases}
\end{align*}
and
\begin{equation*}
\beta_{j,k,i}^t(S) = \beta_{k,j,i}^t(S).
\end{equation*}
\end{proposition}
\begin{proof}
The statement follows from Lemmas~\ref{lem:c2_smooth_vt_bloss_v} and \ref{lem:c2_smooth_vt_bloss_t}.
\end{proof}

\section{Super-Smooth Basis Functions}
\label{sec:ss_basis}

Following the results in Section~\ref{sec:c2_super_smooth}, we now study the $\C{2}$ super-smoothness of the basis functions described in Theorem~\ref{th:c1_basis} and reduce them in a way that they become smoother in certain regions of the domain. An important aspect of these reductions is a possible symmetry of the Powell--Sabin refinement. We say that a triangle $t_{ijk} \in T$ is symmetrically refined if
\begin{itemize}
\item the points $v_k$, $v_{ijk}$, $v_{ij}$ are collinear,
\item the points $v_j$, $v_{ijk}$, $v_{ki}$ are collinear, and
\item the points $v_i$, $v_{ijk}$, $v_{jk}$ are collinear.
\end{itemize}
Let $\Tsym$ denote the subset of $T$ consisting of symmetrically refined triangles. An example of a symmetrically refined triangle is the triangle $t_{ijk}$ shown in Figure~\ref{fig:pstri}. Its neighbouring triangles $t_{ijk'}$ and $t_{kij'}$ are not in $\Tsym$.

\subsection{Basis Functions Associated with Vertices}
\label{sec:basis_vertices}

For $v_i \in V$ and $r \in \set{1,2,3}$, the basis function $B_{i,r}^v$ defined in Section~\ref{sec:c1_basis} is zero on every triangle of $\tri$ that has no vertex at $v_i$. To examine its super-smoothness, it suffices to focus on the triangles of $\tri$ with a vertex at $v_i$.

\begin{proposition}
\label{pro:c2_smoothness_bv}
Let $v_i \in V$ be a vertex of $\tri$ and $r \in \set{1,2,3}$. Suppose that $t_{ijk} \in T$ is a triangle of $\tri$ with a vertex at $v_i$. If $t_{ijk} \in \Tsym$, then the restriction of $B_{i,r}^v$ to $t_{ijk}$ is $\C{2}$ smooth. Otherwise, the restriction of $B_{i,r}^v$ to $t_{ijk}$ is $\C{2}$ smooth across $[v_{ij}, v_{ijk}]$, $[v_{jk}, v_{ijk}]$, $[v_{ki}, v_{ijk}]$ and at $v_{ijk}$.
\end{proposition}
\begin{proof}
For every $t_{ijk} \in T \setminus \Tsym$, the statement follows from \eqref{eq:c1_basis_v}, Proposition~\ref{pro:c2_smoothness_et}, and Remark~\ref{rem:c2_smoothness_et}.
For every $t_{ijk} \in \Tsym$, the statement follows from \eqref{eq:c1_basis_v} and Propositions~\ref{pro:c2_smoothness_et} and~\ref{pro:c2_smoothness_vt}.
\end{proof}

\subsection{Basis Functions Associated with Interior Edges}
\label{sec:basis_e_interior}

In the following, we present techniques for combining basis functions described in Section~\ref{sec:c1_basis}. They allow for a significant reduction of the number of basis functions.

Let $e_{ij} \in E \setminus E^b$ be an interior edge of $\tri$. Suppose that $t_{i,j,k}, t_{j,i,k}, t_{i,j,k'}, t_{j,i,k'} \in \PSR{T}$ are triangles of $\PST$ attached to $e_{ij}$ (see Figure~\ref{fig:pstri} for an illustration). To combine the basis functions $B_{i,j,k}^t$, $B_{j,i,k}^t$, $B_{i,j,k'}^t$, $B_{j,i,k'}^t$, we consider convex weights $\omega_{ij,k}, \omega_{ij,k'} \in [0,1]$ and points $w_{ij,k}, w_{ij,k'}$ such that%
\begin{subequations}%
\label{eq:w_points}
\begin{align}
v_{ijk} &= (1 - \omega_{ij,k}) w_{ij,k} + \omega_{ij,k} w_{ij,k'}, \\
v_{ijk'} &= (1 - \omega_{ij,k'}) w_{ij,k'} + \omega_{ij,k'} w_{ij,k}.
\end{align}
\end{subequations}
With $w_{ij,k}$ and $w_{ij,k'}$ we aim to extend the line segment $[v_{ijk}, v_{ijk'}]$. More precisely, we choose $w_{ij,k} = v_k$ if $t_{ijk} \in \Tsym$ and $w_{ij,k} = v_{ijk}$ otherwise. Similarly, $w_{ij,k'} = v_{k'}$ if $t_{ijk'} \in \Tsym$ and $w_{ij,k'} = v_{ijk'}$ otherwise.

\begin{lemma}
\label{lem:relations_omega_nu}
Let $e_{ij} \in E \setminus E^b$ be an interior edge of $\tri$ and suppose that $v_k$, $v_{ijk}$, $v_{ijk'}$ are collinear and $w_{ij,k} = v_k$. Then,
\begin{equation*}
(1 - \nu_{ij,k}) (1 - \omega_{ij,k}) + \nu_{ij,k} \omega_{ij,k'} = 1,
\end{equation*}
for the weights $\nu_{ij,k}$, $\omega_{ij,k}$, $\omega_{ij,k'}$ introduced in \eqref{eq:nu} and \eqref{eq:w_points}.
\end{lemma}
\begin{proof}
The statement can be verified by comparing the weights in the equation that is obtained by inserting \eqref{eq:w_points} into the expression \eqref{eq:nu} for $v_k$.
\end{proof}

Using the weights $\omega_{ij,k}$ and $\omega_{ij,k'}$, we define the basis functions%
\begin{subequations}%
\label{eq:basis_e_interior}
\begin{align}
B_{ij,k}^e &= (1 - \omega_{ij,k}) \bracket{B_{i,j,k}^t + B_{j,i,k}^t} + \omega_{ij,k'} \bracket{B_{i,j,k'}^t + B_{j,i,k'}^t}, \label{eq:basis_e_interior_t1} \\
B_{ij,k'}^e &= (1 - \omega_{ij,k'}) \bracket{B_{i,j,k'}^t + B_{j,i,k'}^t} + \omega_{ij,k} \bracket{B_{i,j,k}^t + B_{j,i,k}^t}; \label{eq:basis_e_interior_t2}
\end{align}
\end{subequations}
see Figure~\ref{fig:basis_e_interior} for an example.
Notice that when $t_{ijk}$ and $t_{ijk'}$ are not in $\Tsym$, $B_{ij,k}^e$ and $B_{ij,k'}^e$ are simply the sums $B_{i,j,k}^t + B_{j,i,k}^t$ and $B_{i,j,k'}^t + B_{j,i,k'}^t$, respectively.

The basis functions $B_{ij,k}^e$ and $B_{ij,k'}^e$ are zero on every triangle of $\tri$ different from $t_{ijk}$ and $t_{ijk'}$. It suffices to consider $C^2$ super-smoothness only on $t_{ijk}$ as $B_{ij,k}^e$ and $B_{ij,k'}^e$ behave analogously on $t_{ijk}$ and $t_{ijk'}$.

\begin{figure}[t!]
\centering
\includegraphics[width = 0.49\textwidth]{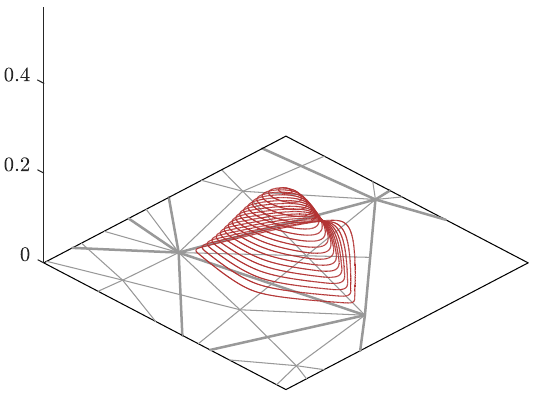}~
\includegraphics[width = 0.49\textwidth]{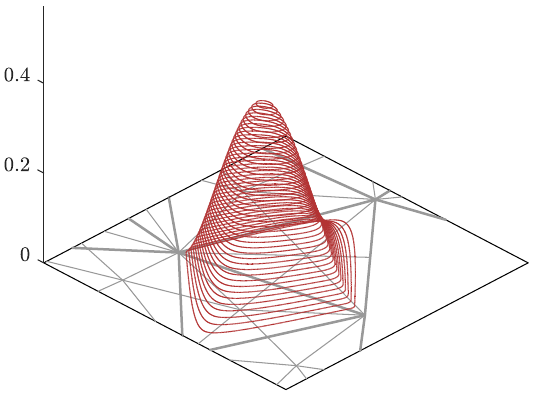}
\caption{The basis functions $B_{ij,k}^e$ (left) and $B_{ij,k'}^e$ (right) obtained by combining the basis functions $B_{i,j,k}^t$, $B_{j,i,k}^t$, $B_{i,j,k'}^t$, $B_{j,i,k'}^t$ shown in Figure~\ref{fig:c1_basis_t}.}
\label{fig:basis_e_interior}
\end{figure}

\begin{proposition}
\label{pro:c2_smoothness_be}
Let $e_{ij} \in E \setminus E^b$ be an interior edge of $\tri$. Suppose that $t_{ijk}, t_{ijk'} \in T$ are the triangles of $\tri$ attached to the edge $e_{ij}$. Then, the restrictions of $B_{ij,k}^e$ and $B_{ij,k'}^e$ to $t_{ijk}$ are $\C{2}$ smooth across $[v_{ij}, v_{ijk}]$, $[v_{jk}, v_{ijk}]$, $[v_{ki}, v_{ijk}]$ and at $v_{ijk}$. In addition, if $t_{ijk} \in \Tsym$, then the restriction of $B_{ij,k'}^e$ to $t_{ijk}$ is $\C{2}$ smooth.
\end{proposition}
\begin{proof}
The $\C{2}$ smoothness of $B_{ij,k}^e$ and $B_{ij,k'}^e$ across $[v_{ij}, v_{ijk}]$, $[v_{jk}, v_{ijk}]$, $[v_{ki}, v_{ijk}]$ and at $v_{ijk}$ follows from \eqref{eq:c1_basis_t}, Proposition~\ref{pro:c2_smoothness_et}, and Remark~\ref{rem:c2_smoothness_et}.
When $t_{ijk} \in \Tsym$, the $\C{2}$ smoothness of $B_{ij,k'}^e$ across $[v_i, v_{ijk}]$, $[v_j, v_{ijk}]$, $[v_k, v_{ijk}]$ follows from Proposition~\ref{pro:c2_smoothness_vt} by applying \eqref{eq:c1_basis_t} and Lemma~\ref{lem:relations_omega_nu}.
\end{proof}

Contrarily to $B_{ij,k'}^e$, the property $t_{ijk} \in \Tsym$ does not imply that the basis function $B_{ij,k}^e$ restricted to $t_{ijk}$ is $\C{2}$ smooth across $[v_i, v_{ijk}]$. Indeed, it follows from \eqref{eq:c1_basis_t} and Lemmas~\ref{lem:c2_smooth_vt_bloss_v} and \ref{lem:relations_omega_nu} that%
\begin{subequations}%
\label{eq:blossom_inter_be}
\begin{align}
\blossr{B_{ij,k}^e}{t_{i,j,k}}(v_i, v_j, v_k) &= 1, \label{eq:blossom_inter_be_t1} \\
\blossr{B_{ij,k}^e}{t_{i,k,j}}(v_i, v_j, v_k) &= 0, \label{eq:blossom_inter_be_t2}
\end{align}
\end{subequations}
which violates the $\C{2}$ smoothness condition given in \eqref{eq:c2_smooth_blossom_vt_v}.
This motivates a further recombination of the basis functions described in Section~\ref{sec:basis_t}.

\subsection{Basis Functions Associated with Boundary Edges}

Let $e_{ij} \in E^b$ be a boundary edge of $\tri$. Analogous to the definitions in Section~\ref{sec:basis_e_interior}, we now consider a convex weight $\omega_{ij,k} \in [0,1]$ and a point $w_{ij,k}$ such that
\begin{equation}
\label{eq:w_point_boundary}
v_{ijk} = (1 - \omega_{ij,k}) w_{ij,k} + \omega_{ij,k} v_{ij}.
\end{equation}
We choose $w_{ij,k} = v_k$ if $t_{ijk} \in \Tsym$ and $w_{ij,k} = v_{ijk}$ otherwise.

\begin{lemma}
\label{lem:relations_omega_nu_boundary}
Let $e_{ij} \in E^b$ be a boundary edge and suppose that $v_k$, $v_{ijk}$, $v_{ij}$ are collinear and $w_{ij,k} = v_k$. Then,
\begin{equation*}
(1 - \nu_{ij,k}) (1 - \omega_{ij,k}) = 1,
\end{equation*}
for the weights $\nu_{ij,k}$ and $\omega_{ij,k}$ introduced in \eqref{eq:nu} and \eqref{eq:w_point_boundary}.
\end{lemma}
\begin{proof}
The statement can be verified by comparing the weights in the equation that is obtained by inserting \eqref{eq:w_point_boundary} into the expression \eqref{eq:nu} for $v_k$.
\end{proof}

Using the weight $\omega_{ij,k}$, we define the basis functions%
\begin{subequations}%
\label{eq:basis_e_boundary}
\begin{align}
B_{ij,k}^e &= (1 - \omega_{ij,k}) \bracket{B_{i,j,k}^t + B_{j,i,k}^t}, \label{eq:basis_e_boundary_t} \\
B_{ij}^e &= \bracket{B_{i,j}^e + B_{j,i}^e} + \omega_{ij,k} \bracket{B_{i,j,k}^t + B_{j,i,k}^t}; \label{eq:basis_e_boundary_e}
\end{align}
\end{subequations}
see Figure~\ref{fig:basis_e_boundary} for an example. Notice that when $t_{ijk}$ is not in $\Tsym$, $B_{ij,k}^e$ and $B_{ij}^e$ are simply the sums $B_{i,j,k}^t + B_{j,i,k}^t$ and $B_{i,j}^e + B_{j,i}^e$, respectively.

We focus the investigation of $\C{2}$ super-smoothness on the triangle $t_{ijk}$ as both $B_{ij,k}^e$ and $B_{ij}^e$ are zero on every triangle of $\tri$ different from $t_{ijk}$.

\begin{figure}[t!]
\centering
\includegraphics[width = 0.49\textwidth]{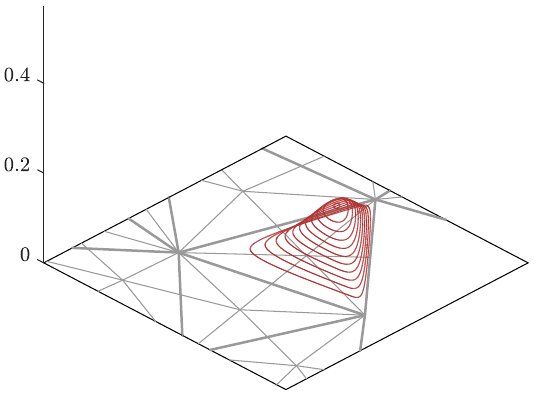}~
\includegraphics[width = 0.49\textwidth]{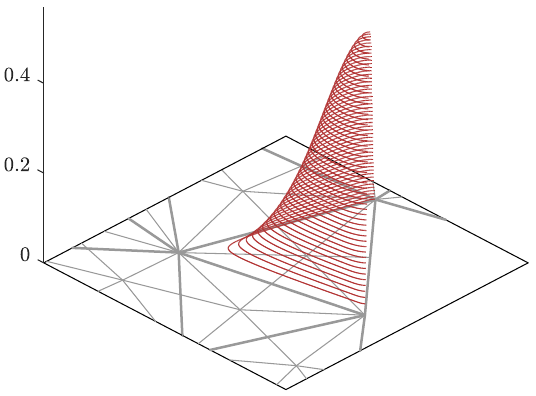}
\caption{The basis functions $B_{jk,i}^e$ (left) and $B_{jk}^e$ (right) obtained by combining the basis functions $B_{j,k,i}^t$, $B_{k,j,i}^t$, $B_{j,k}^e$, $B_{k,j}^e$ shown in Figure~\ref{fig:c1_basis_e}.}
\label{fig:basis_e_boundary}
\end{figure}

\begin{proposition}
\label{pro:c2_smoothness_bbe}
Let $e_{ij} \in E^b$ be a boundary edge of $\tri$. Suppose that $t_{ijk} \in T$ is the triangle of $\tri$ attached to $e_{ij}$. Then, the restrictions of $B_{ij,k}^e$ and $B_{ij}^e$ to $t_{ijk}$ are $\C{2}$ smooth across $[v_{ij}, v_{ijk}]$, $[v_{jk}, v_{ijk}]$, $[v_{ki}, v_{ijk}]$ and at $v_{ijk}$. In addition, if $t_{ijk} \in \Tsym$, then the restriction of $B_{ij}^e$ to $t_{ijk}$ is $\C{2}$ smooth.
\end{proposition}
\begin{proof}
The $\C{2}$ smoothness of $B_{ij,k}^e$ and $B_{ij}^e$ across $[v_{ij}, v_{ijk}]$, $[v_{jk}, v_{ijk}]$, $[v_{ki}, v_{ijk}]$ and at $v_{ijk}$ follows from \eqref{eq:c1_basis_t}, \eqref{eq:c1_basis_e}, Proposition~\ref{pro:c2_smoothness_et}, and Remark~\ref{rem:c2_smoothness_et}. When $t_{ijk} \in \Tsym$, the $\C{2}$ smoothness of $B_{ij}^e$ across $[v_i, v_{ijk}]$, $[v_j, v_{ijk}]$, $[v_k, v_{ijk}]$ follows from Proposition~\ref{pro:c2_smoothness_vt} by applying \eqref{eq:c1_basis_t}, \eqref{eq:c1_basis_e}, and Lemma~\ref{lem:relations_omega_nu_boundary}.
\end{proof}

Again, as in Section~\ref{sec:basis_e_interior}, the property $t_{ijk} \in \Tsym$ does not imply that the basis function $B_{ij,k}^e$ restricted to $t_{ijk}$ is $\C{2}$ smooth across $[v_i, v_{ijk}]$. It follows from \eqref{eq:c1_basis_t} and Lemmas~\ref{lem:c2_smooth_vt_bloss_v} and \ref{lem:relations_omega_nu_boundary} that%
\begin{subequations}%
\label{eq:blossom_bound_be}
\begin{align}
\blossr{B_{ij,k}^e}{t_{i,j,k}}(v_i, v_j, v_k) &= 1, \label{eq:blossom_bound_be_t1} \\
\blossr{B_{ij,k}^e}{t_{i,k,j}}(v_i, v_j, v_k) &= 0, \label{eq:blossom_bound_be_t2}
\end{align}
which violates the $\C{2}$ smoothness condition given in \eqref{eq:c2_smooth_blossom_vt_v}.
\end{subequations}

\subsection{Basis Functions Associated with Triangles}
\label{sec:basis_t}

For a triangle $t_{ijk} \in \Tsym$, the basis functions can be further reduced. We define
\begin{equation}
\label{eq:basis_t}
B_{ijk}^t = B_{ij,k}^e + B_{jk,i}^e + B_{ki,j}^e;
\end{equation}
see Figure~\ref{fig:basis_t} for an example. 

The basis function $B_{ijk}^t$ is supported on the region determined by the triangle $t_{ijk}$ and the triangles of $\tri$ that share an edge with $t_{ijk}$. We now consider $\C{2}$ super-smoothness properties of $B_{ijk}^t$ on this region.

\begin{figure}[t!]
\centering
\includegraphics[width = 0.49\textwidth]{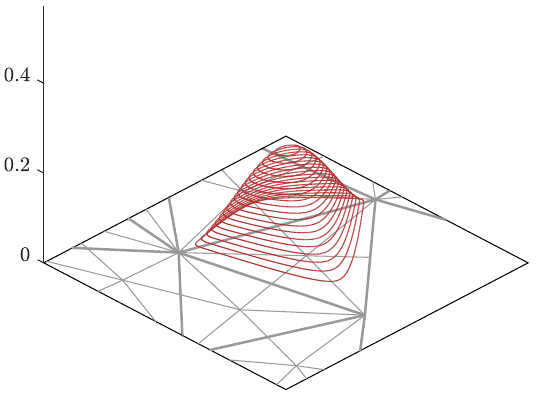}~
\includegraphics[width = 0.49\textwidth]{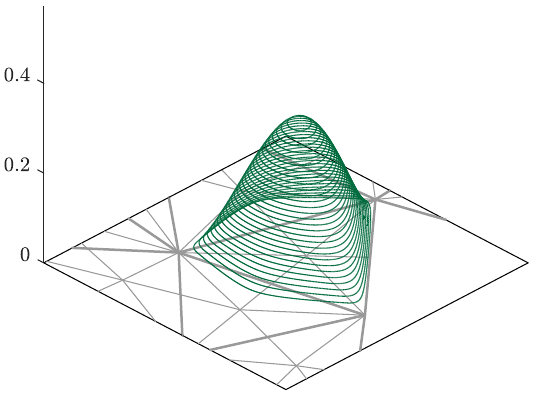}
\caption{The basis function $B_{ki,j}^e$ (left) and the basis function $B_{ijk}^t$ (right) obtained by combining the basis functions $B_{ij,k}^e$, $B_{jk,i}^e$ (shown in Figures \ref{fig:basis_e_interior} and \ref{fig:basis_e_boundary}), and $B_{ki,j}^e$.}
\label{fig:basis_t}
\end{figure}

\begin{proposition}
\label{pro:c2_smoothness_t_bt}
Let $t_{ijk} \in \Tsym$ be a triangle of $\tri$. Then, the restriction of $B_{ijk}^t$ to $t_{ijk}$ is $\C{2}$ smooth.
\end{proposition}
\begin{proof}
The $\C{2}$ smoothness of $B_{ijk}^t$ across $[v_{ij}, v_{ijk}]$, $[v_{jk}, v_{ijk}]$, $[v_{ki}, v_{ijk}]$ follows from Propositions~\ref{pro:c2_smoothness_be} and \ref{pro:c2_smoothness_bbe}.
The $\C{2}$ smoothness of $B_{ijk}^t$ across $[v_i, v_{ijk}]$, $[v_j, v_{ijk}]$, $[v_k, v_{ijk}]$ follows from Proposition~\ref{pro:c2_smoothness_vt} by applying \eqref{eq:c1_basis_t} and Lemmas~\ref{lem:relations_omega_nu} and~\ref{lem:relations_omega_nu_boundary}.
\end{proof}

\begin{proposition}
\label{pro:c2_smoothness_nt_bt}
Let $t_{ijk} \in \Tsym$ be a triangle of $\tri$, $e_{ij} \in E \setminus E^b$ an interior edge of $\tri$, and $t_{ijk'} \in T$ the second triangle of $\tri$ attached to $e_{ij}$. If $t_{ijk'} \in \Tsym$, then the restriction of $B_{ijk}^t$ to $t_{ijk'}$ is $\C{2}$ smooth. Otherwise, the restriction is $\C{2}$ smooth across $[v_{ij}, v_{ijk'}]$, $[v_{jk'}, v_{ijk'}]$, $[v_{k'i}, v_{ijk'}]$ and at $v_{ijk'}$.
\end{proposition}
\begin{proof}
The statement follows from Proposition~\ref{pro:c2_smoothness_be}.
\end{proof}

\section{A Reduced Spline Space}
\label{sec:reduced_space}

In this section we explain how the reduced basis functions introduced in Section~\ref{sec:ss_basis} can be arranged to span a subspace of $\splspace{3}{1}{\PST}$ that maintains cubic precision. We investigate the representation of splines from the considered subspace and derive their main properties that make them attractive for use in approximation methods and geometric design. We also outline their main advantages compared to the basis of $\splspace{3}{1}{\PST}$ described in Theorem~\ref{th:c1_basis}.

\subsection{Basis Functions}

We start by describing a set comprising $3|V| + |\Tsym| + 3 |T \setminus \Tsym| + |E^b|$ spline functions. To identify the $ 3|T \setminus \Tsym|$ parameters, we introduce the following notation. For an edge $e_{ij} \in E$, let $T_{ij}^e \subseteq T$ denote the set of triangles attached to $e_{ij}$. The set $T_{ij}^e$ has either one element (if $e_{ij} \in E^b$) or two elements (if $e_{ij} \in E \setminus E^b$). Let $\redsplspaceb \subset \splspace{3}{1}{\PST}$ be defined as follows.
\begin{itemize}
\item For each $v_i \in V$ and $r \in \set{1,2,3}$, the set $\redsplspaceb$ contains the function $B_{i,r}^v$ defined by \eqref{eq:c1_basis_v}.

\item For each $t_{ijk} \in \Tsym$, the set $\redsplspaceb$ contains the function $B_{ijk}^t$ defined by \eqref{eq:basis_t}.

\item For each $e_{ij} \in E$ and $t_{ijk} \in T_{ij}^e \setminus \Tsym$, the set $\redsplspaceb$ contains the function $B_{ij,k}^e$ defined by \eqref{eq:basis_e_interior_t1} if $e_{ij} \in E \setminus E^b$ or \eqref{eq:basis_e_boundary_t} if $e_{ij} \in E^b$.

\item For each $e_{ij} \in E^b$, the set $\redsplspaceb$ contains the function $B_{ij}^e$ defined by \eqref{eq:basis_e_boundary_e}.
\end{itemize}
We denote by $\redsplspace{3}{1}{\PST}$ the vector space spanned by the functions in $\redsplspaceb$. The following theorems list some of its properties.

\begin{theorem}
The space $\redsplspace{3}{1}{\PST}$ is a subspace of $\splspace{3}{1}{\PST}$.
\end{theorem}
\begin{proof}
The functions from $\redsplspaceb$ are linear combinations of the basis functions of $\splspace{3}{1}{\PST}$ presented in Theorem~\ref{th:c1_basis}.
\end{proof}

\begin{theorem}
A spline $S \in \redsplspace{3}{1}{\PST}$ possesses the following $\C{2}$ smoothness properties.
\begin{itemize}
\item For every triangle $t_{ijk} \in \Tsym$ of $\tri$, the restriction of $S$ to $t_{ijk}$ is $\C{2}$ smooth.
\item For every triangle $t_{ijk} \in T \setminus \Tsym$ of $\tri$, the restriction of $S$ to $t_{ijk}$ is $\C{2}$ smooth across $[v_{ij}, v_{ijk}]$, $[v_{jk}, v_{ijk}]$, $[v_{ki}, v_{ijk}]$ and at $v_{ijk}$.
\end{itemize}
\end{theorem}
\begin{proof}
The statement follows from Propositions~\ref{pro:c2_smoothness_bv}--\ref{pro:c2_smoothness_nt_bt}.
\end{proof}

\subsection{Spline Representation}

With the functions from $\redsplspaceb$, we associate the following functionals acting on $S \in \redsplspace{3}{1}{\PST}$.
\begin{itemize}
\item For each $v_i \in V$ and $r \in \set{1,2,3}$, let $\beta_{i,r}^v(S)$ be defined as in \eqref{eq:beta_v}.
\item For each $t_{ijk} \in \Tsym$, let $\beta_{ijk}^t(S) = \blossr{S}{t_{i,j,k}}(v_i, v_j, v_k)$.
\item For each $e_{ij} \in E$ and $t_{ijk} \in T_{ij}^e \setminus \Tsym$, let $\beta_{ij,k}^e(S) = \blossr{S}{t_{i,j,k}}(v_i, v_j, v_{ijk})$.
\item For each $e_{ij} \in E^b$, let $\beta_{ij}^e(S) = \blossr{S}{e_{i,j}}(v_i, v_j, v_{ij})$.
\end{itemize}
Note that by Lemma~\ref{lem:c2_smooth_vt_bloss_v} we have
\begin{equation*}
\beta_{ijk}^t(S) = (1 - \nu_{ij,k}) \beta_{i,j,k}^t(S) + \nu_{ij,k}
\begin{cases}
\beta_{i,j,k'}^t(S) & \text{if} \ e_{ij} \in E \setminus E^b, \\
\beta_{i,j}^e(S) & \text{if} \ e_{ij} \in E^b,
\end{cases}
\end{equation*}
and by Proposition~\ref{pro:c2_smoothness_et},
\begin{equation*}
\beta_{ij,k}^e(S) = \beta_{i,j,k}^t(S) = \beta_{j,i,k}^t(S), \qquad
\beta_{ij}^e(S) = \beta_{i,j}^e(S) = \beta_{j,i}^e(S).
\end{equation*}
We now provide relations between these functionals and the considered basis functions.

\begin{lemma}
\label{lem:biorthogonality_v}
Let $B \in \redsplspaceb$. For each $v_i \in V$ and $r \in \set{1,2,3}$,
\begin{equation*}
\beta_{i,r}^v(B) =
\begin{cases}
1 & \text{if} \ B = B_{i,r}^v, \\
0 & \text{otherwise}.
\end{cases}
\end{equation*}
\end{lemma}
\begin{proof}
Since $B_{i,r}^v \in \splspaceb$ and $\beta_{i,r}^v \in \splspacedb$, the property \eqref{eq:c1_basis_v} implies  $\beta_{i,r}^v(B_{i,r}^v) = 1$. Moreover, since every $B \in \redsplspaceb \setminus \tset{B_{i,r}^v}$ is a linear combination of the basis functions in $\splspaceb \setminus \tset{B_{i,r}^v}$, the properties in \eqref{eq:c1_basis} imply $\beta_{i,r}^v(B) = 0$.
\end{proof}

\begin{lemma}
\label{lem:biorthogonality_t}
Let $B \in \redsplspaceb$. For each $t_{ijk} \in \Tsym$,
\begin{equation*}
\beta_{ijk}^t(B) =
\begin{cases}
1 & \text{if} \ B = B_{ijk}^t, \\
0 & \text{otherwise}.
\end{cases}
\end{equation*}
\end{lemma}
\begin{proof}
By \eqref{eq:blossom_inter_be_t1} and \eqref{eq:blossom_bound_be_t1}, the value of $\beta_{ijk}^t(B_{ij,k}^e)$ is equal to $1$. Additionally, by Lemma~\ref{lem:c2_smooth_vt_bloss_v}, the definitions of $B_{jk,i}^e$ and $B_{ki,j}^e$ analogous to \eqref{eq:basis_e_interior_t1}, and \eqref{eq:c1_basis_t}, the values of $\beta_{ijk}^t(B_{jk,i}^e)$ and $\beta_{ijk}^t(B_{ki,j}^e)$ are equal to $0$. In view of \eqref{eq:basis_t}, this confirms that $\beta_{ijk}^t(B_{ijk}^t) = 1$. If $B \in \redsplspaceb \setminus \tset{B_{ijk}^t}$, then \eqref{eq:blossom_inter_be_t2}, \eqref{eq:blossom_bound_be_t2}, Lemmas~\ref{lem:c2_smooth_vt_bloss_v}, \ref{lem:relations_omega_nu}, \ref{lem:relations_omega_nu_boundary}, and the properties in \eqref{eq:c1_basis} imply $\beta_{ijk}^t(B) = 0$.
\end{proof}

\begin{lemma}
\label{lem:biorthogonality_e}
Let $B \in \redsplspaceb$. For each $e_{ij} \in E$ and $t_{ijk} \in T_{ij}^e \setminus \Tsym$,
\begin{equation*}
\beta_{ij,k}^e(B) =
\begin{cases}
1 & \text{if} \ B = B_{ij,k}^e, \\
0 & \text{otherwise}.
\end{cases}
\end{equation*}
\end{lemma}
\begin{proof}
Let us first suppose that $e_{ij} \in E \setminus E^b$, i.e., $e_{ij}$ is an interior edge. By \eqref{eq:c1_basis_t}, applying $\beta_{ij,k}^e$ to \eqref{eq:basis_e_interior_t1} yields $1 - \omega_{ij,k}$. Since $t_{ijk} \not\in \Tsym$, the value of $\omega_{ij,k}$ is equal to $0$, and hence $\beta_{ij,k}^e(B_{ij,k}^e) = 1$. Similarly, for the second triangle $t_{ijk'} \in T_{ij}^e$, applying $\beta_{ij,k}^e$ to \eqref{eq:basis_e_interior_t2} yields $\beta_{ij,k}^e(B_{ij,k'}^e) = \omega_{ij,k} = 0$. Since $\beta_{ij,k}^e(B_{jk',i}^e) = 0$ and $\beta_{ij,k}^e(B_{k'i,j}^e) = 0$, this implies $\beta_{ij,k}^e(B_{ijk'}^t) = 0$ if $t_{ijk'} \in \Tsym$. Any other $B \in \redsplspaceb$ is a linear combination of the functions from $\splspaceb \setminus \tset{B_{i,j,k}^t, B_{j,i,k}^t, B_{i,j,k'}^t, B_{j,i,k'}^t}$, and hence \eqref{eq:c1_basis} implies $\beta_{ij,k}^e(B) = 0$.

If $e_{ij} \in E^b$, i.e., $e_{ij}$ is a boundary edge, then $\beta_{ij,k}^e(B_{ij,k}^e) = 1 - \omega_{ij,k}$ and $\beta_{ij,k}^e(B_{ij}^e) = \omega_{ij,k}$ by \eqref{eq:basis_e_boundary}, \eqref{eq:c1_basis_t}, and \eqref{eq:c1_basis_e}. Since $t_{ijk} \not\in \Tsym$, $\beta_{ij,k}^e(B_{ij,k}^e) = 1$ and $\beta_{ij,k}^e(B_{ij}^e) = 0$. Any other $B \in \redsplspaceb$ is a linear combination of the functions from $\splspaceb \setminus \tset{B_{i,j,k}^t, B_{j,i,k}^t, B_{i,j}^e, B_{j,i}^e}$, and hence \eqref{eq:c1_basis} implies $\beta_{ij,k}^e(B) = 0$.
\end{proof}

\begin{lemma}
\label{lem:biorthogonality_be}
Let $B \in \redsplspaceb$. For each $e_{ij} \in E^b$,
\begin{equation*}
\beta_{ij}^e(B) =
\begin{cases}
1 & \text{if} \ B = B_{ij}^e, \\
0 & \text{otherwise}.
\end{cases}
\end{equation*}
\end{lemma}
\begin{proof}
By \eqref{eq:basis_e_boundary}, \eqref{eq:c1_basis_t}, and \eqref{eq:c1_basis_e}, we obtain $\beta_{ij}^e(B_{ij}^e) = 1$ and $\beta_{ij}^e(B_{ij,k}^e) = 0$.  Since $\beta_{ij}^e(B_{jk,i}^e) = 0$ and $\beta_{ij}^e(B_{ki,j}^e) = 0$, this implies $\beta_{ij}^e(B_{ijk}^t) = 0$ if $t_{ijk} \in \Tsym$. Any other $B \in \redsplspaceb$ is a linear combination of the functions from $\splspaceb \setminus \tset{B_{i,j,k}^t, B_{j,i,k}^t, B_{i,j}^e, B_{j,i}^e}$, and hence \eqref{eq:c1_basis} implies $\beta_{ij}^e(B) = 0$.
\end{proof}

The following theorem provides a characterization of $\redsplspace{3}{1}{\PST}$.
\begin{theorem}
\label{th:red_basis}
The set $\redsplspaceb$ is a basis of $\redsplspace{3}{1}{\PST}$. Every $S \in \redsplspace{3}{1}{\PST}$ can be uniquely expressed as
\begin{align*}
S =& \sum_{v_i \in V} \sum_{r = 1}^3 \beta_{i,r}^v(S) B_{i,r}^v + \sum_{t_{ijk} \in \Tsym} \beta_{ijk}^t(S) B_{ijk}^t \\
&+ \sum_{e_{ij} \in E} \ \sum_{t_{ijk} \in T_{ij}^e \setminus \Tsym} \beta_{ij,k}^e(S) B_{ij,k}^e + \sum_{e_{ij} \in E^b} \beta_{ij}^e(S) B_{ij}^e.
\end{align*}
\end{theorem}
\begin{proof}
The statement follows from Lemmas~\ref{lem:biorthogonality_v}--\ref{lem:biorthogonality_be}.
\end{proof}

\subsection{Properties of the Reduced Basis Functions}

The reduction of $\splspaceb$ to $\redsplspaceb$ is designed in a way that the functions in $\redsplspaceb$ preserve the main properties of the functions comprising $\splspaceb$. For optimal approximation, a crucial property is the reproduction of polynomials to the highest possible degree.

\begin{theorem}
\label{th:poly_reproduction}
The space $\pp_3$ is a subspace of $\redsplspace{3}{1}{\PST}$.
\end{theorem}
\begin{proof}
Suppose $P \in \pp_3$. Clearly, $P \in \splspace{3}{1}{\PST}$, and hence it can be expressed as described in Theorem~\ref{th:c1_basis}. Let $e_{ij} \in E$ be an edge of $\tri$. We consider only the case when $e_{ij} \in E \setminus E^b$, i.e., $e_{ij}$ is an interior edge of $\tri$, as the reasoning for a boundary edge is analogous. Let $t_{ijk}, t_{ijk'} \in T_{ij}^e$ be the triangles of $\tri$ attached to $e_{ij}$. By \eqref{eq:beta_t}, $\beta_{i,j,k}^t(P) = \beta_{j,i,k}^t(P)$ and $\beta_{i,j,k'}^t(P) = \beta_{j,i,k'}^t(P)$. Then,
\begin{align*}
\beta_{i,j,k}^t(P) B_{i,j,k}^t + \beta_{j,i,k}^t(P) B_{j,i,k}^t &= \bloss{P}(v_i, v_j, v_{ijk}) \bracket{B_{i,j,k}^t + B_{j,i,k}^t}, \\
\beta_{i,j,k'}^t(P) B_{i,j,k'}^t + \beta_{j,i,k'}^t(P) B_{j,i,k'}^t &= \bloss{P}(v_i, v_j, v_{ijk'}) \bracket{B_{i,j,k'}^t + B_{j,i,k'}^t}.
\end{align*}
Moreover, we can use \eqref{eq:w_points} and the affinity of the blossom in the third argument to express
\begin{align*}
\bloss{P}(v_i, v_j, v_{ijk}) &= (1 - \omega_{ij,k}) \bloss{P}(v_i, v_j, w_{ij,k}) + \omega_{ij,k} \bloss{P}(v_i, v_j, w_{ij,k'}), \\
\bloss{P}(v_i, v_j, v_{ijk'}) &= (1 - \omega_{ij,k'}) \bloss{P}(v_i, v_j, w_{ij,k'}) + \omega_{ij,k'} \bloss{P}(v_i, v_j, w_{ij,k}).
\end{align*}
Hence, 
$$\beta_{i,j,k}^t(P) B_{i,j,k}^t + \beta_{j,i,k}^t(P) B_{j,i,k}^t + \beta_{i,j,k'}^t(P) B_{i,j,k'}^t + \beta_{j,i,k'}^t(P) B_{j,i,k'}^t$$ 
is by \eqref{eq:basis_e_interior} equal to
\begin{equation*}
\bloss{P}(v_i, v_j, w_{ij,k}) B_{ij,k}^e + \bloss{P}(v_i, v_j, w_{ij,k'}) B_{ij,k'}^e.
\end{equation*}
Since
\begin{equation*}
\bloss{P}(v_i, v_j, w_{ij,k}) =
\begin{cases}
\beta_{ijk}^t(P) & \text{if} \ t_{ijk} \in \Tsym, \\
\beta_{ij,k}^e(P) & \text{if} \ t_{ijk} \in T \setminus \Tsym,
\end{cases}
\end{equation*}
and
\begin{equation*}
\bloss{P}(v_i, v_j, w_{ij,k'}) =
\begin{cases}
\beta_{ijk'}^t(P) & \text{if} \ t_{ijk'} \in \Tsym, \\
\beta_{ij,k'}^e(P) & \text{if} \ t_{ijk'} \in T \setminus \Tsym,
\end{cases}
\end{equation*}
the statement follows from the definitions of $B_{ijk}^t$ and $B_{ijk'}^t$.
\end{proof}

The functions from $\redsplspaceb$ also possess the B-spline properties observed for $\splspaceb$.

\begin{theorem}
The functions in $\redsplspaceb$ are locally supported, nonnegative, and form a partition of unity. Moreover, the polynomials $(x,y) \mapsto x$ and $(x,y) \mapsto y$ can be represented as
\begin{align*}
(x,y) =& \sum_{v_i \in V} \sum_{r = 1}^3 q_{i,r}^v B_{i,r}^v(x,y) + \sum_{t_{ijk} \in \Tsym} q_{ijk}^t B_{ijk}^t(x,y) \\
&+ \sum_{e_{ij} \in E} \ \sum_{t_{ijk} \in T_{ij}^e \setminus \Tsym} q_{ij,k}^e B_{ij,k}^e(x,y) + \sum_{e_{ij} \in E^b} q_{ij}^e B_{ij}^e(x,y),
\end{align*}
where $q_{i,r}^v$ is defined in Section~\ref{sec:c1_basis} and
\begin{align*}
q_{ijk}^t &= \tfrac{1}{3} v_i + \tfrac{1}{3} v_j + \tfrac{1}{3} v_k, \\
q_{ij,k}^e &= \tfrac{1}{3} v_i + \tfrac{1}{3} v_j + \tfrac{1}{3} v_{ijk}, \\
q_{ij}^e &= \tfrac{1}{3} v_i + \tfrac{1}{3} v_j + \tfrac{1}{3} v_{ij}.
\end{align*}
\end{theorem}
\begin{proof}
The properties of local support and nonnegativity for the functions in $\redsplspaceb$ follow from the way these functions are defined in \eqref{eq:basis_e_interior}, \eqref{eq:basis_e_boundary}, \eqref{eq:basis_t} and the fact that the functions in $\splspaceb$ possess these properties. The partition of unity and the representation of the linear polynomials are corollaries of Theorems~\ref{th:red_basis} and \ref{th:poly_reproduction} and the properties of the blossom. More precisely, $\bloss{1}(p_1, p_2, p_3) = 1$ and $(\bloss{x}(p_1, p_2, p_3), \bloss{y}(p_1, p_2, p_3)) = \frac{1}{3} p_1 + \frac{1}{3} p_2 + \frac{1}{3} p_3$ for any three points $p_1, p_2, p_3 \in \RR^2$.
\end{proof}

\subsection{Comparison of the Number of Degrees of Freedom}

The advantage of the basis $\redsplspaceb$ is that it provides the same quality of approximation as $\splspaceb$ with a considerably smaller number of degrees of freedom. To argue that, let us approximately compare the dimensions of $\redsplspace{3}{1}{\PST}$ and $\splspace{3}{1}{\PST}$. We use the estimates $|E| \approx 3 |V|$ and $|T| + |E^b| \approx 2|V|$; see, e.g., \cite{lai_07}.

The dimension of $\splspace{3}{1}{\PST}$ is $3|V| + 4|E| \approx 15 |V|$. On the other hand, the dimension of $\redsplspace{3}{1}{\PST}$ depends on $|\Tsym|$, the number of symmetrically refined triangles. If $\Tsym = \emptyset$, then the dimension of $\redsplspace{3}{1}{\PST}$ is $3|V| + 2|E| \approx 9 |V|$. In this case, the space corresponds to the Powell--Sabin spline space studied in \cite{ps3_speleers_15} and the basis $\redsplspaceb$ matches the one proposed therein. If $\Tsym = T$, then the dimension of $\redsplspace{3}{1}{\PST}$ is $3|V| + |T| + |E^b| \approx 5|V|$. If the triangulation $\tri$ is three-directional, the basis $\redsplspaceb$ agrees with the constructions presented in \cite{ps3_groselj_21}, except for the basis functions associated with the boundary edges.

\begin{figure}[t!]
\centering
\includegraphics[width=0.48\textwidth]{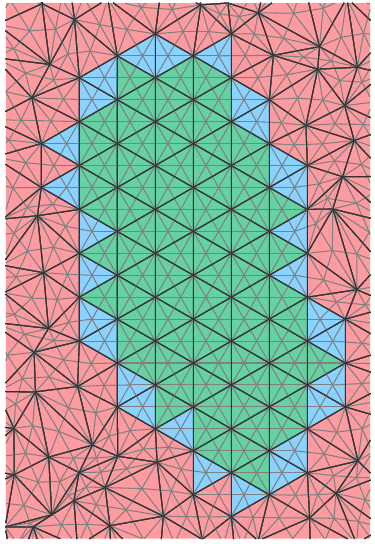}~
\includegraphics[width=0.48\textwidth]{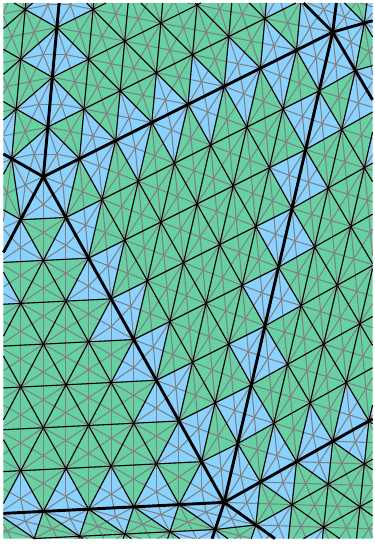}
\caption{Excerpts from two partially structured triangulations refined with a Powell--Sabin refinement. On the left is a triangulation, a part of which is three-directional. On the right is a triangulation obtained by uniformly splitting an initial unstructured triangulation (depicted with thicker lines).}
\label{fig:pstri_example}
\end{figure}

\begin{example}
Let us consider the reduced spline space over a partially structured triangulation, an excerpt of which is shown in Figure~\ref{fig:pstri_example} (left). The super-smooth basis functions on the region determined by red colored triangles are partially $\C{2}$ smooth in the interior of these triangles and match the basis functions introduced in \cite{ps3_speleers_15}. The green color indicates symmetrically refined triangles and so in their interior the super-smooth basis functions are $\C{2}$ smooth. Since in this case the triangulation determined by the green colored triangles is three-directional, the basis functions correspond to the basis functions investigated in \cite{ps3_groselj_21} (with a possible exception of the basis functions associated with the vertices that are not shared only by green colored triangles). The triangles colored in blue represent the transition area. Following the previous discussion, we can expect an approximate reduction in the number of basis functions by a factor between $\frac{5}{3}$ and $3$ compared with the number of functions in the basis $\splspaceb$ of the full $\C{1}$ spline space.
\end{example}

\begin{example}
Starting from an arbitrary triangulation, application of one or more uniform refinements  results in (local) symmetric configurations inside each of the original triangles; see Figure~\ref{fig:pstri_example} (right) for an example. The green colored triangles are symmetrically refined, and the blue colored triangles are part of the transition area. Such refined triangulations were analyzed in \cite{ps3_groselj_23}. The corresponding (reduced) cubic spline spaces were tested in the context of least squares approximation and finite element approximation for second and fourth order boundary value problems. Note that the definition of the set of triangles that are part of the symmetric configurations in \cite{ps3_groselj_23}, denoted by $T_{\mathrm{sym}}$, is more restrictive than the definition of $\Tsym$ used here. Indeed, it holds $T_{\mathrm{sym}}\subseteq\Tsym$.
\end{example}

\section{Conclusion}
\label{sec:conclusion}

In this paper, we have analyzed $\C{2}$ super-smoothness conditions of a previously established $\C{1}$ cubic spline space on a Powell--Sabin refined triangulation. This has been done by employing a B-spline basis for $\C{1}$ splines with a convenient dual representation developed in \cite{ps3_groselj_17}. The representation is based on blossoming of polynomials obtained by restricting a spline to the triangles of the Powell--Sabin refinement.

Blossoming offers an elegant and concise way of expressing the smoothness conditions between polynomial joints. Here this technique has been utilized to identify the $\C{2}$ smoothness conditions between the functionals of the dual basis. Some of these conditions can be enforced without difficulty on general triangulations. Others are more involved but greatly simplify if the triangulation and its corresponding Powell--Sabin refinement possess certain symmetries.

It has also been shown how the $\C{2}$ smoothness constraints can be integrated into the spline representation by reducing the set of basis functions. The proposed super-smooth B-spline basis functions are obtained by recombination of the original $\C{1}$ basis functions and incorporate the $\C{2}$ smoothness conditions to an extent that depends on the geometry of the underlying refinement.

The reduced cubic spline space obtained as the span of the constructed super-smooth B-spline basis functions contains all polynomials up to degree three and maintains the cubic precision of the full $\C{1}$ spline space. Such space offers the usual flexibility of unstructured spline methods for handling irregularly partitioned parts of the domain, while the partial structuredness of the triangulation, which is common in practice, is exploited by reducing the number of degrees of freedom. 

Starting from an arbitrary triangulation, application of one or more uniform refinements results in (local) symmetric configurations inside each of the original triangles. The implementation of reduced cubic spline spaces on such triangulations was investigated in \cite{ps3_groselj_23}. It was tested in the context of least squares approximation and finite element approximation, and the optimal convergence order was illustrated for those reduced spaces. In this paper, we have provided a more general framework to achieve space reductions by imposing super-smoothness. As such, we envision it as a versatile tool for approximation purposes.

\section*{Acknowledgements}
J.~Gro\v{s}elj was partially supported by the research programme P1-0294 of Javna agencija za znanstvenoraziskovalno in inovacijsko dejavnost Republike Slovenije (ARIS). 
H.~Speleers was supported in part by a GNCS 2022 project (CUP E55F22000270001) of Gruppo Nazionale per il Calcolo Scientifico -- Istituto Nazionale di Alta Matematica (GNCS -- INdAM), by the MUR Excellence Department Project MatMod@TOV (CUP E83C23000330006) awarded to the Department of Mathematics of the University of Rome Tor Vergata, and by the National Research Center in High Performance Computing, Big Data and Quantum Computing (CUP E83C22003230001).

\bibliography{references}

\end{document}